\documentclass[a4paper,11pt]{llncs}
\usepackage[utf8]{inputenc}
\usepackage[T1]{fontenc}
\usepackage{zi4}
\usepackage[a4paper,margin=2.5cm]{geometry}

\usepackage{cite}
\usepackage{mathtools,amsfonts}
\usepackage{mathptmx}
\usepackage[english]{babel}
\usepackage{enumerate}
\usepackage{amsmath}
\usepackage{amssymb,latexsym}
\usepackage{varioref}
\usepackage{hyperref}
\usepackage{cleveref}
\usepackage{xcolor}
\usepackage{xspace}
\usepackage{csquotes}
\usepackage[super]{nth}
\usepackage{url}

\hypersetup{
	linktoc = page,
	pdfpagemode = UseNone,
	colorlinks,
	linkcolor={red!50!black},
	citecolor={blue!50!black},
	urlcolor={blue!80!black}
}

\usepackage{tikz}
\usepackage{pgfplots}
\pgfplotsset{compat=1.14}

\usepackage{booktabs}
\newcommand{\dig}[1]{\mbox{\texttt{#1}}}
\newcommand{\rl}{\mathsf{RunL}}
\newcommand{\lsa}{\mathsf{LSA}}

\renewenvironment*{example}[0]{\refstepcounter{example}\smallskip\noindent\textbf{Example~\theexample.\xspace}}{}

\DeclareFontFamily{U}{mathb}{\hyphenchar\font45}
\DeclareFontShape{U}{mathb}{m}{n}{
      <5> <6> <7> <8> <9> <10> gen * mathb
      <10.95> mathb10 <12> <14.4> <17.28> <20.74> <24.88> matha12
      }{}
\DeclareSymbolFont{mathb}{U}{mathb}{m}{n}

\DeclareMathSymbol{\downuparrows}{3}{mathb}{"D7}
\DeclareMathSymbol{\bft}{3}{mathb}{"FD}

\begin{document}

\title{Stuttering Conway Sequences Are Still Conway Sequences}

\author{%
    \'Eric Brier\inst{2} \and Rémi Géraud-Stewart\inst{1} \and David Naccache\inst{1}\and Alessandro Pacco\inst{1} \and Emanuele Troiani\inst{1}
} 

\institute{
    ÉNS (DI), Information Security Group,
CNRS, PSL Research University, 75005, Paris, France.\\
45 rue d'Ulm, 75230, Paris \textsc{cedex} 05, France\\
\email{\url{remi.geraud@ens.fr}},~~~~  
\email{\url{given\_name.family\_name@ens.fr}}
	\and
	Ingenico Laboratories, 75015, Paris, France.\\
	\email{\url{eric.brier@ingenico.com}}\\
}
\maketitle

\begin{abstract}
A look-and-say sequence is obtained iteratively by reading off the digits of the current value, grouping identical digits together: starting with \dig{1}, the sequence reads: \dig{1}, \dig{11}, \dig{21}, \dig{1211}, \dig{111221}, \dig{312211}, etc. (\textsf{OEIS A005150}).
Starting with any digit $d \neq {\dig{1}}$ gives Conway's sequence: $d$, ${\dig{1}}d$, ${\dig{111}}d$, ${\dig{311}}d$, ${\dig{13211}}d$, etc. (\textsf{OEIS A006715}). Conway popularised these sequences and studied some of their properties\cite{conway1987weird}.

In this paper we consider a variant subbed \enquote{\emph{look-and-say again}} where digits are repeated twice. We prove that the \enquote{\emph{look-and-say again}} sequence contains only the digits ${\dig{1}},{\dig{2}},{\dig{4}},{\dig{6}},d$, where $d$ represents the starting digit. Such sequences decompose and the ratio of successive lengths converges to Conway's constant.

In fact, these properties result from a commuting diagram between look-and-say again sequences and \enquote{classical} look-and-say sequences. Similar results apply to the \enquote{look-and-say three times} sequence.
\end{abstract}

\section{Introduction}

The look-and-say (\textsf{LS}) sequence \cite{conway2012book}, also known as the \emph{Morris} or the \emph{Conway sequence} \cite{conway1987weird,hilgemeier1996one,ekhad1997proof} is a recreational integer sequence having very intriguing properties.\smallskip

A \textsf{LS} sequence is obtained iteratively by reading off the digits of the current value, and counting the number of digits in groups of the identical digit.\smallskip

Starting with \texttt{1}, the sequence reads (\textsf{OEIS A005150}): \texttt{1}, \texttt{11}, \texttt{21}, \texttt{1211}, \texttt{111221}, \texttt{312211}, etc. Starting with any digit $d \neq {\dig{1}}$ gives Conway's sequence (\textsf{OEIS A006715}): $d$, ${\dig{1}}d$, ${\dig{111}}d$, ${\dig{311}}d$, ${\dig{13211}}d$, etc. Conway popularised these sequences and studied some of their properties. For example, an \textsf{LS} sequence contains only the digits $\dig{1}, \dig{2}, \dig{3}$, and satisfies a so-called \emph{cosmological decay} \cite{ekhad1997proof}, if $L_n$ denotes the number of digits of the $n^{\mbox{\scriptsize{th}}}$ term of the sequence, then 
\begin{equation*}
\lim _{n\to \infty }{\frac {L_{n+1}}{L_{n}}}=\lambda\simeq 1.303577
\end{equation*}
where $\lambda$ is the only real root of a degree-$71$ polynomial \cite[§6.12]{finch2003mathematical}. 
Conway showed that $\exists N\in\mathbb{N}^*$ such that every term of the \textsf{LS} term decays in at most $N$ rounds to a compound of \enquote{common} and \enquote{transuranic} terms.

Following Conway's work, numerous variants of \textsf{LS} sequences were proposed and studied.
For instance, Pea Pattern sequences \cite{pea}, Sloane's sequences\cite{slo} or Kolakoski sequences \cite{kola,kola2}.
In this paper we consider a new \textsf{LS} sequence and study some of its properties.
The concerned variant, called \textit{\enquote{look-and-say again}} sequence, consists in repeating each \textsf{LS} digit twice. We prove that the such sequences contain only the digits ${\dig{1}},{\dig{2}},{\dig{4}},{\dig{6}},d$, where $d$ is the starting digit.

\section{Notations and definitions}
In this paper we assume that numbers are written in base 10. Any integer $T$ can thus be written $T = t_1t_2 \cdots t_k$ with $t_1, \dotsc, t_k \in \{{\dig{0}}, {\dig{1}}, \dotsc, {\dig{9}}\}$. To avoid any ambiguity, $ab$ will denote the \emph{concatenation} of the numbers $a$ and $b$; accordingly $a^b$ indicates that a digit $a$ is repeated $b$ times. If we want to emphasise concatenation we use $a\|b$ instead of $ab$.

\begin{definition}[Run-length representation]
Let $T\in \mathbb{N}^*$, we can write 
\begin{equation*}
    T = \underbrace{a_1\ldots a_1}_{n_1}\underbrace{a_2\ldots a_2}_{n_2}\ldots\underbrace{a_k\ldots a_k}_{n_k}
\end{equation*}
with $a_1\neq a_2, a_2\neq a_3, \dotsc, a_{k-1}\neq a_{k}$. The \emph{run-length representation} of $T$ is the sequence $\rl(T)= a_1^{n_1}a_2^{n_2} \cdots a_k^{n_k}$. Conversely, any finite sequence of couples $(a_i, b_i)_i$ where $a \in \mathbb N^*$ and ${\dig{0}}\leq b_i\leq {\dig{9}}$ is such that $b_{i-1}\neq b_i\neq b_{i+1}$, corresponds to an integer with run-length representation $(a_i^{b_i})_i$. 
\end{definition}
Note that the run-length representation of an integer is unique. 

\begin{definition}[Pieces]
If $N = (a_i^{b_i})$ is a run-length encoded integer, we call each $a_i^{b_i}$ a \emph{piece} of $N$.
\end{definition}

\begin{definition}[Look-and-say-again sequence]
Let $T_0$ be a decimal digit, and for each $T_n$ define
\begin{equation*}
    T_{n+1} = n_1n_1a_1a_1n_2n_2a_2a_2\cdots n_kn_ka_ka_k.
\end{equation*}
where $(a_i^{n_i})_i = \rl(T_n)$. We call the sequence $(T_k)_{k \in \mathbb N}$ the \emph{look-and-say-again sequence} of seed $T_0$, and denote it by $\lsa(T_0)$.
\end{definition}

\begin{example}

\noindent
$
\lsa(1)=
\small{
    {\dig{1}}\to
    {\dig{1111}}\to
    {\dig{4411}}\to
    {\dig{22442211}}\to
    {\dig{2222224422222211}}\to
    {\dig{6622224466222211}}\to \dotsc
}$

\noindent
$\lsa(2)= 
\small{
{\dig{2}}\to
{\dig{1122}}\to
{\dig{22112222}}\to
{\dig{222222114422}}\to
{\dig{6622221122442222}}\to
{\dig{226644222211222222444422}}\to
\dotsc}$
\end{example}

\section{The look-and-say-again sequence}

\begin{theorem}[Digits of $\lsa$]
Only the digits $\dig{1}, \dig{2}, \dig{3}, \dig{4}, \dig{6}, d$ appear in $\lsa(d)$.
\end{theorem}

\begin{proof}
Let $n \in \mathbb N$ and $(a_i^{n_i})_i = \rl(n)$, we write
\begin{equation*}
P(n):=\{\forall i, \,a_i \in \{\dig{1}, \dig{2}, \dig{4}, \dig{6}\} \text{ and } n_i \in \{2, 4, 6\}\}
\end{equation*}
(the matter of $d$ will be settled further down). Assume that $P(T_n)$ is true, and let $a_i^{n_i} \in \rl(T_n)$.
We have four situations:
\begin{enumerate}[\text{Case} 1:]
    \item  $\rl(T_n) = a_i^{n_i}$, in other terms $T_n = a_ia_i \cdots a_i$ and there is no other digit. Then the next term in the sequence is $T_{n+1}=n_in_ia_ia_i$, which clearly satisfies $P(T_{n+1})$ since $a_i\in\{\dig{1},\dig{2},\dig{4}, \dig{6}\}$ and $n_i\in\{2, 4, 6\}$. 

\item $i=1$, i.e., $T_n$ starts with the repeated digits $a_i$. In this case
\begin{equation*}
T_{n+1}=n_1n_1a_1a_1n_{2}n_{2}a_{2}a_{2} \cdots
\end{equation*}
with $a_k\neq a_{k+1}$ for all $k$. It is clear that no $a_{k}$ can be contained in a piece that also contains $a_{k+1}$, therefore the possible pieces that $n_1$ and $a_1$ can take part to are either $n_1n_1a_1a_1n_{2}n_{2}$ (in the case $n_1=a_1=n_{2}$), or $n_1n_1a_1a_1$ (in the case $n_1=a_1\neq n_{2}$), or $n_1n_1$ and $a_1a_1$ / $a_1a_1n_{2}n_{2}$ (respectively in the cases $n_1\neq a_1\neq n_{2}$ and $n_1\neq a_1=n_{2}$). 

In conclusion $a_1$ and $n_1$ generate pieces made of digits in $\{\dig{1},\dig{2},\dig{4}, \dig{6}\}$ with multiplicity either 2, 4, or 6. 
\item $T_n$ ends with $a_i$; this is analogous to case 2 above.

\item The piece $a_i^{n_i}$ is neither at the end nor at the beginning of $T_n$. The next term in the sequence is:
\begin{equation*}
T_{n+1}=\cdots a_{i-2}n_{i-1}n_{i-1} a_{i-1}a_{i-1}n_in_ia_ia_in_{i+1}n_{i+1}a_{i+1}a_{i+1}\cdots
\end{equation*}
By definition of the run-length representation, $a_{i-2},a_{i-1},a_i,a_{i+1}$ contains no consecutive values, hence the possible pieces resulting for $n_i$ and $a_i$ are:
\begin{enumerate}
    \item $n_{i-1}n_{i-1}a_{i-1}a_{i-1}n_{i}n_i$, and either $a_ia_i$ or $a_ia_in_{i+1}n_{i+1}$
    \item $a_{i-1}a_{i-1}n_{i}n_i$, and either $a_ia_i$ or $a_ia_in_{i+1}n_{i+1}$
    \item $n_{i}n_i$, and either $a_ia_i$ or $a_ia_in_{i+1}n_{i+1}$
    \item $n_in_ia_ia_i$ or $n_in_ia_ia_in_{i+1}n_{i+1}$
    \item $n_in_i$ and $a_ia_i$
\end{enumerate}
 In each case, since $\forall k, \,a_k\in\{\dig{1},\dig{2},\dig{4},\dig{6}\}$ and $n_k\in\{2, 4, 6\}$,
 we see that $n_i$ and $a_i$ can appear only in pieces that are of multiplicity 2, 4, or 6, and which contain numbers $\in\{\dig{1},\dig{2},\dig{4},\dig{6}\}$.
\end{enumerate}
If $T_0 = d \in \{\dig{1},\dig{2},\dig{4},\dig{6}\}$, then $P(T_0)$ holds and by the above case exhaustion argument $P(T_n)$ hold for all $n$.

It remains to discuss the case $T_0 = d \notin \{\dig{1},\dig{2},\dig{4},\dig{6}\}$. Writing the first few terms of the resulting sequence shows that this is easily dealt with:
 \begin{align*}
      T_0 & =  d 
      & T_1 & = \dig{11}dd 
      & T_2 & = \dig{221122}dd
      & T_3 & = \dig{22222211222222}dd & \cdots
 \end{align*}
 Indeed, save for the first term, the digit $d$ only appears as $dd$ at the end of $T_n$. The rest of $T_n$ satisfies $P(T_n)$ discussed previously.
 
 To prove this, assume that $T_n = k \| dd$, which means that $T_n$ starts with an integer $k$ and ends with the two digits $dd$, and further assume that $P(k)$ is true and $k$'s last piece is $\dig{2}^k$, with $k\in\{2,4,6\}$. 
 Then the next term in the sequence is $T_{n+1} = k' \| \dig{22}dd$, where $k'$ is an integer that ends with the digit \dig{2} and such that $P(k')$ is true (thanks to what we have proved in the first part of the theorem). Let $S = k'\|\dig{22}$. Consider three cases, as a function of the last piece of $k$, denoted $\dig{2}^{\omega}$:
 \begin{center}
\begin{tabular}{ccccccr}
$\omega={2}$&$\Rightarrow k'$ ends with &$\dig{2222}$ 
&$\Rightarrow S$ ends with & $\dig{222222}$&$\Rightarrow P(S) \mbox{~holds}$\\
$\omega={4}$&$\Rightarrow k'$ ends with &$\dig{4422}$ 
&$\Rightarrow S$ ends with & $\dig{2222}$&$\Rightarrow P(S) \mbox{~holds}$\\
$\omega={6}$&$\Rightarrow k'$ ends with &$\dig{6622}$ 
&$\Rightarrow S$ ends with & $\dig{2222}$&$\Rightarrow P(S) \mbox{~holds}$\\
 \end{tabular}
\end{center}
 Since $P(k')$ holds, the only problem in $S$ was at the interface between the ending \dig{2}'s of $k'$ and the couple \dig{22} at the end of $S$. With the exhaustion argument above we have shown that in each possible case $P(S)$ holds. Finally the number $S\|dd$ is such that $d$ only appears as a couple $dd$ at the end, whereas $S$ is made of digits belonging to $\{\dig{1},\dig{2},\dig{4},\dig{6}\}$ with multiplicities in $\{2, 4, 6\}$.
 
$T_0$ and $T_1$ contain only digits $\in\{\dig{1}, \dig{2}, \dig{4}, \dig{6},d\}$. Since $T_2$ is of the form $k\|dd$ with $P(k)$ true and $k$'s last digit being $\dig{2}$, the above argument shows that all subsequent terms in the sequence can be written in this way, and the proof is completed.\qed 
\end{proof}

\begin{corollary}
If $T_0 = d \neq \dig{1}, \dig{2}$ then $\lsa(T_0)$ gives the same sequence, save for the two last digits of each term which are $dd$.
\end{corollary}

\begin{remark}
The length sequences for $T_0 = \dig{1}, \dig{2}, d$ are respectively:
\begin{align*}
    & 1, 4, 4, 8, 16, 16, 24, 40, 48, 56, 88, 104, 120, 176, 224, 280, 392, 520, 648, 864, 1168 \\
    & 1, 4, 8, 12, 16, 24, 32, 40, 48, 72, 92, 112, 156, 204, 264, 352, 464, 592, 784, 1036, 1320 \\
    & 1, 4, 8, 16, 16, 24, 40, 48, 56, 88, 104, 120, 176, 224, 280, 392, 520, 648, 864, 1168, 1432
\end{align*}
\end{remark}

\begin{remark}
For all seeds $s$, $\lsa(d)$ grows to infinity, namely the ratio of lengths $\lambda_n = \ell_{n+1}/\ell_n$ for two consecutive terms of sequence (which is between $1$ and $2$) tends towards Conway's constant $\lambda \simeq 1.303577$, regardless of the seed $T_0$. The following numerical simulation backs up this intuition:
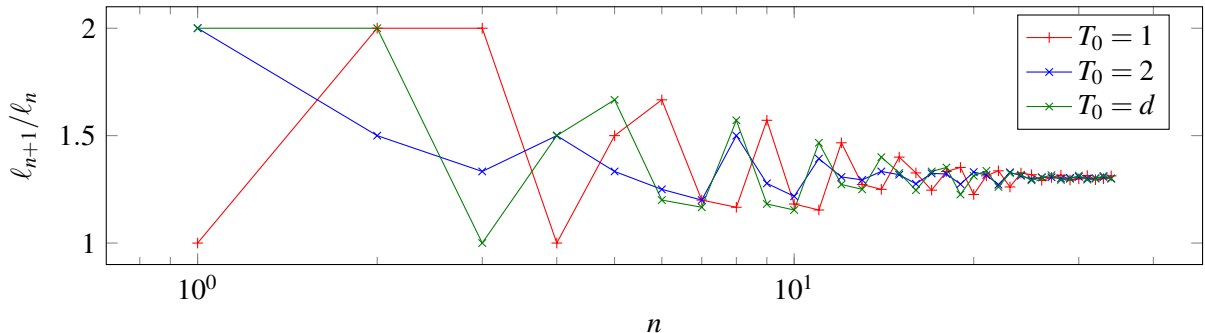
\begin{center}
\begin{tikzpicture}
\begin{semilogxaxis}[legend cell align = left,legend pos = north east,xlabel=$n$,ylabel=$\ell_{n+1}/\ell_n$,width=\textwidth,height=5cm]
\addplot[color=red, mark=+] table[x=i,y=r1] {ratios.dat};
\addlegendentry{$T_0 = 1$};
\addplot[color=blue, mark=x] table[x=i,y=r2] {ratios.dat};
\addlegendentry{$T_0 = 2$};
\addplot[color=green!50!black, mark=x] table[x=i,y=r3] {ratios.dat};
\addlegendentry{$T_0 = d$};
\end{semilogxaxis}
\end{tikzpicture}
\end{center}
\end{remark}
This is in fact a consequence of the following result:
\begin{theorem}
Consider the following operations on pieces:
$$
C: a^b \mapsto ba ~~~~~~~~~~ L: a^b \mapsto b^2 a^2 ~~~~~~~~~~ \eta: a^b \mapsto \kappa(a)^{2b}
$$
with $\kappa: a \mapsto \dig{2}a$ for $a\in\{\dig{1}, \dig{2}, \dig{3}\}$ and $\kappa(a) = \dig{1}$ otherwise. Then $L \circ \eta = \eta \circ C$.
\end{theorem}

\begin{proof}
Let $a^b$ be a piece,
\begin{align*}
    (\eta \circ C)(a^b) & = \eta(b^1 a^1) = \eta(b^1) \eta(a^1) = \kappa(b)^2 \kappa(a)^2 \\
    (L \circ \eta)(a^b) & = L(\kappa(a)^{\kappa(b)}) = \kappa(b)^2 \kappa(a)^2.
\end{align*}
\qed 
\end{proof}
A result of this theorem is that $\lsa$ is equivalent to Conway's sequence, where $\kappa$ and $\eta$ allow us to translate from one to the other. Then $\lsa$ inherits many of the properties that are known of Conway's sequence: decomposition into \enquote{elements}, convergence to $\lambda$, and so forth.
\begin{remark}
This would also work with $L(a^b) = b^3 a^3$ and $\kappa(a) = \dig{3}a$ for $a \in \{\dig{1}, \dig{2}, \dig{3}\}$ (i.e. an $\lsa$ variant where elements are repeated three times instead of two). However the argument breaks down for a \enquote{look-and-say four times} sequence. We leave this sequence and the study of its properties open for further research.
\end{remark}

\bibliographystyle{alpha}
\bibliography{conway.bib}

\begin{thebibliography}{{Mul}12}

\bibitem[CG12]{conway2012book}
John~H Conway and Richard Guy.
\newblock {\em The book of numbers}.
\newblock Springer Science \& Business Media, 2012.

\bibitem[Con87]{conway1987weird}
John~H Conway.
\newblock The weird and wonderful chemistry of audioactive decay.
\newblock In {\em Open problems in communication and computation}, pages
  173--188. Springer, 1987.

\bibitem[EZ97]{ekhad1997proof}
Shalosh Ekhad and Doron Zeilberger.
\newblock {Proof of Conway's lost cosmological theorem}.
\newblock {\em Electronic Research Announcements of the American Mathematical
  Society}, 3(11):78--82, 1997.

\bibitem[Fin03]{finch2003mathematical}
Steven~R Finch.
\newblock {\em Mathematical constants}.
\newblock Cambridge university press, 2003.

\bibitem[Hil96]{hilgemeier1996one}
Mario Hilgemeier.
\newblock {One metaphor fits all: A fractal voyage with Conway's audioactive
  decay}, 1996.

\bibitem[Kol66]{kola}
William Kolakoski.
\newblock Problem 5304.
\newblock In {\em American Mathematical Monthly}, volume~72, page 674. 1966.

\bibitem[{Mul}12]{pea}
{Multiple authors}.
\newblock {Ascending Pea Pattern generator}, 2012.
\newblock \url{https://tinyurl.com/REF-PEA12}.

\bibitem[Slo09]{slo}
NJA Sloane.
\newblock Seven staggering sequences.
\newblock {\em Homage to a Pied Puzzler, E. Pegg Jr., AH Schoen and T. Rodgers
  (editors), AK Peters, Wellesley, MA}, pages 93--110, 2009.

\bibitem[\"U66]{kola2}
Necdet \"U\c{c}oluk.
\newblock {Self Generating Runs}.
\newblock In {\em American Mathematical Monthly}, volume~73, page 681–682.
  1966.

\end{thebibliography}
\end{document}